\documentclass[a4paper,12pt]{article}

\usepackage[utf8]{inputenc}
\usepackage[english]{babel}
\usepackage{enumitem}
\usepackage{amssymb}
\usepackage{amsmath}
\usepackage{amsthm}
\usepackage{mathrsfs}
\usepackage{amsfonts}
\usepackage{array}
\usepackage{babel}
\usepackage{booktabs}
\usepackage{multirow}
\usepackage{epsfig}
\usepackage{graphicx}
\usepackage{subfig}
\usepackage{epstopdf}
\usepackage{amsfonts}
\usepackage{enumerate}
\usepackage{color}
\usepackage{cite}

\usepackage{caption}
\newtheorem{dfn}{Definition}[section]
\newtheorem{lem}[dfn]{Lemma}

\newtheorem{alg}[dfn]{Algorithm}
\newtheorem{theorem}{Theorem}[section]
\theoremstyle{definition}

\newtheorem{asm}[dfn]{Assumption}
\newtheorem{exm}[dfn]{Example}
\newtheorem{rem}[dfn]{Remark}

\DeclareMathOperator*{\zer}{zer}

\DeclareMathOperator*{\gra}{gra}

\DeclareMathOperator*{\nt}{int}
\DeclareMathOperator{\dom}{dom}
\DeclareMathOperator{\Pro}{Prob}
\allowdisplaybreaks

\title{Stochastic forward-backward-half forward splitting algorithm with variance reduction}
\author{ \sc \normalsize Liqian Qin$^{a,b}${\thanks{ email: qlqmath@163.com}},\,\,
Yaxuan Zhang$^{a}${\thanks{email: bunnyxuan@tju.edu.cn}},\,\,
Qiao-Li Dong$^{a}${\thanks{Corresponding author. email: dongql@lsec.cc.ac.cn}}
\,\,\,and\,
Michael Th. Rassias$^{c,d}${\thanks{email: mthrassias@yahoo.com}}\\
\small $^{a}$College of Science,  Civil Aviation University of China, Tianjin 300300, China,\\
\small $^{b}$School of Mathematics and Information Science, Guangzhou University,\\
 \small  Guangzhou 510006, China,\\
\small $^{c}$Department of Mathematics and Engineering Sciences,
Hellenic Military Academy, \\
\small 16673 Vari Attikis, Greece,\\
\small $^{d}$Program in Interdisciplinary Studies, 1 Einstein Dr,  Princeton, NJ 08540, USA.
}
\date{}

\begin{document}

	\maketitle

	\begin{abstract}
\noindent
\\
In this paper, we present a stochastic forward-backward-half forward splitting algorithm with variance reduction for solving the structured monotone inclusion problem composed of a maximally monotone operator, a maximally monotone operator and a cocoercive operator in a separable
real Hilbert space. By defining a Lyapunov function, we establish the weak almost sure convergence of the proposed algorithm, and obtain the linear convergence when one of the maximally monotone operators is strongly monotone. Numerical examples are provided to show the performance of the proposed algorithm.\\


\end{abstract}

\section{Introduction and preliminaries}
\noindent
In this paper, we consider the structured monotone inclusion problem in a separable real Hilbert space $\mathcal{H}$ which is to find $x\in \mathcal{H}$  such that
\begin{eqnarray}
\label{PROB}
 \ 0\in (A+B+C)(x),
\end{eqnarray}
where $A: \mathcal{H} \rightarrow 2^{\mathcal{H}}$ is a maximally monotone set-valued operator, $B: \mathcal{H} \rightarrow \mathcal{H}$ is a monotone point-valued operator, and $C: \mathcal{H} \rightarrow \mathcal{H}$ is a $\beta$-cocoercive operator. Problem \eqref{PROB} arises in various applications such as optimization problems \cite{FBHF,{Davis}}, deep learning \cite{Barnet}, image deblurring \cite{Tang},  variational inequalities \cite{Alacaoglu-Malitsky}, equilibrium problems and games \cite{{Quoc},{Shehu},{Thong}}.
For example, the user equilibrium traffic assignment problem can be formulated as a finite dimensional variational inequality: find $x^* \in \Omega,$ such that
\begin{equation}
\label{variantion}
\langle F(x^*),x-x^*\rangle\geq 0, \quad \forall x \in \Omega.
\end{equation}
Let $A$ be the subdifferential of the indicator function of a closed convex set $\Omega \subset \mathcal{H},$ $B+C=F,$ then problem \eqref{variantion} becomes the problem \eqref{PROB}.

Numerous iterative algorithms for solving \eqref{PROB} have been presented and analyzed, see, for instance, \cite{{Comb},{Davis},{Latafat},{Malitsky},{Rie},{Ryu},{Ryuu},{Yu},{Zong},{Tang}} and references therein.  In particular, Brice\~{n}o-Arias et al. \cite{FBHF} first proposed a forward-backward-half forward (FBHF) splitting algorithm as follows
\begin{equation}
\label{FBHF}
\left\{
\begin{array}{lr}
p^{k}=J_{\gamma^{k} A}\left(x^{k}-\gamma^{k}(B+C)x^{k}\right), & \\
x^{k+1}=P_{X}(p^{k}+\gamma^{k}(Bx^{k}-Bp^{k})), & \\
\end{array}
\right.
\end{equation}
where $\gamma^{k}$ is step-size, $\gamma^{k}\in [\eta , \chi-\eta]$, $\eta \in (0,\frac{\chi}{2})$, $\chi=\frac{4 \beta}{1+\sqrt{1+16\beta^2L_B^2}}$, $J_{\gamma^{k} A}=({\rm Id}+\gamma^{k} A)^{-1}$ is the resolvent of $A$, and $X$ is a nonempty closed convex subset of $\mathcal{H}$ containing a solution of the problem \eqref{PROB}. They obtained the weak convergence of the method \eqref{FBHF} in a real Hilbert space.

In many cases, monotone inclusion problems have a finite sum structure. For example, finite sum minimization is ubiquitous in machine learning where we minimize the  empirical risk \cite{LJC}, and nonlinear constrained optimization problems \cite{FBHF}. Finite sum saddle-point problems and
finite sum variational inequalities can also be transformed into a monotone inclusion
problems \cite{Zhang}.
Given the effectiveness of variance-reduced algorithms for finite sum function minimization, a natural idea is to use similar algorithms to solve the more general finite sum monotone inclusion problems.

Now, we detail our problem setting. Suppose that the maximally monotone operator $B$ in \eqref{PROB} has a finite sum representation $B= \sum_{i=1}^N B_i$, where each $B_i$ is $L_i$-Lipschitz. Then the problem \eqref{PROB} can be written in the following form
\begin{equation}\label{PROB2}
\mbox{Find} \ x \in \mathcal{H}  \ \ \mbox{such that} \ 0\in (A+\sum_{i=1}^N B_i+C)(x).
\end{equation}
If $B$ is $L_B$-Lipschitz, then it  might be the case that $L_i$ are easy to compute, but not $L_B$. In this case, $\sum_{i=1}^N L_i \geq L_B$ gives us a most natural upper bound on $L_B$. On the other hand, the cost of computing $Bx$ is rather expansive  when $N$ is very large.

Throughout this paper, we assume access to a stochastic oracle $B_{\xi}$ such that  $B_{\xi}$ is unbiased, $B(x)=\mathbb{E}[B_{\xi}(x)]$, and then consider utilizing the stochastic oracle $B_{\xi}$ to perform in the half forward step in the \eqref{FBHF} instead of $B$, which yields lower cost per iteration. We also assume that $B$ is $L$-Lipschitz in mean. The two simplest stochastic oracles can be defined as follows
\begin{itemize}
\item[(i)] Uniform sampling: $B_{\xi}(x)=NB_i(x)$, $P_{\xi}(i)=\Pro \{\xi=i\}=\frac{1}{N}$. In this case, $L=\sqrt{N\sum_{i=1}^N L_i^2}$.
\item[(ii)] Importance sampling: $B_{\xi}(x)=\frac{1}{P_{\xi}(i)}B_i(x)$, $ P_{\xi}(i)=\Pro\{\xi=i\}=\frac{L_i}{\sum_{j=1}^N L_j}$. In this case, $L=\sum_{i=1}^N L_i$.
\end{itemize}

Recently, Kovalev et al. \cite{KHR} proposed a loopless variant of stochastic variance reduced gradient (SVRG) \cite{SVRG}
which removes the outer loop present in SVRG and uses a probabilistic update of the full gradient instead. Later, Alacaoglu et al. \cite{Alacaoglu-Malitsky} proposed the loopless version of extragradient method with variance reduction for solving variational inequalities. They also applied the same idea over the forward-backward-forward (FBF) splitting algorithm which was introduced by Tseng \cite{Tseng} to solve the two operators monotone inclusion problem in a finite dimensional Euclidean space,
\begin{eqnarray*}
\mbox{find} \ x \in \mathbb{R}^d \ \ \mbox{such that} \ 0\in (A+B)(x),
\end{eqnarray*}
where $A: \mathbb{R}^d \rightarrow 2^{\mathbb{R}^d}$ and $B: \mathbb{R}^d \rightarrow \mathbb{R}^d$ are maximally monotone operators. The operator $B: \mathbb{R}^d \rightarrow \mathbb{R}^d$ has a stochastic oracle $B_{\xi}$ that is unbiased and Lipschitz in mean. They proved the almost sure convergence of the forward-backward-forward splitting algorithm with variance reduction when $B_{\xi}$ is continuous for all $\xi$. However,  the cocoercive operator $C$ is required to admit a finite-sum structure as well, if one extends the forward-backward-forward splitting algorithm with variance reduction to solve problem \eqref{PROB2}.


In this paper, we propose a stochastic forward-backward-half forward splitting algorithm with variance reduction (shortly, VRFBHF).  Under some mild assumptions, we establish the weak almost sure convergence of the sequence $\{x^k\}_{k \in \mathbb{N}}$ generated by our algorithm. Lyapunov analysis of the proposed algorithm is based on the monotonicity inequalities of $A$ and $B$, and the cocoercivity inequality of $C$. Furthermore, we obtain the linear convergence when $A$ or $B$ is strongly monotone. Numerical experiments are conducted to demonstrate the efficacy of the proposed algorithm.

Next, we recall some definitions and known results which will be helpful for further analysis.

Throughout this paper, $\mathcal{H}$ is a separable Hilbert space with inner product  $\langle \cdot, \cdot\rangle$, induced norm $\|\cdot\|$, {and Borel $\sigma$-algebra $\mathcal{B}$.} $\mathbb{R}^d$ is a $d$-dimensional Euclidean space. The set of nonnegative integers is denoted by $\mathbb{N}$. {The probability space is $(\Omega, \mathcal{F}, P).$ A $\mathcal{H}$-valued random variable is a measurable map $x: (\Omega, \mathcal{F})\rightarrow (\mathcal{H}, \mathcal{B}).$ The $\sigma$-algebra
generated by a family $\Phi$ of random variables is denoted by $\sigma(\Phi).$ Let $\mathscr{F}= \{\mathcal{F}_k\}_{k\in \mathbb{N}}$
be a sequence of sub-sigma algebras of $\mathcal{F}$ such that $\mathcal{F}_k \subset \mathcal{F}_{k+1}.$ } Probability
mass function $P_{\xi}(\cdot)$ is supported on $\{1, \ldots, N\}$.  We denote the strong convergence and weak convergence by $``\rightarrow"$ and $``\rightharpoonup"$, respectively.

\begin{dfn}\rm{(\cite[Definition 20.1 and Definition 20.20]{BC2011}\ )}
{\rm
 A set-valued mapping $A: \mathcal{H} \rightarrow 2^{\mathcal{H}}$ is characterized by its graph $\gra(A)=\{(x,u) \in \mathcal{H}\times\mathcal{H}:u \in Ax\}.$
A set-valued mapping $A: \mathcal{H} \rightarrow 2^{\mathcal{H}}$ is said to be\\
{\rm (i)}  monotone if $\langle u-v,x-y \rangle \geq 0$ for all $(x,u),(y,v) \in \gra(A).$\vskip 1mm
\noindent
{\rm (ii)}  maximally monotone if there exists no monotone operator $B: \mathcal{H} \rightarrow 2^{\mathcal{H}}$ such that $\gra(B)$ properly contains $\gra(A),$ i.e., for every $(x,u) \in \mathcal{H}\times \mathcal{H}$,
$$
(x,u) \in \gra (A)  \ \ \Leftrightarrow \ \ \langle u-v , x-y \rangle \geq 0,  \ \  \forall(y,v)\in \gra (A).
$$
}
\end{dfn}

\begin{dfn}
{\rm
An operator  $T: \mathcal{H} \rightarrow \mathcal{H}$ is said to be\vskip 1mm
\noindent
{\rm (i)} $L$-Lipschitz continuous, if there exists a constant $ L > 0$, such that
$$
\|Tx-Ty\|\leq  L\|x-y\|,\quad\forall x, y\in \mathcal{H};
$$\vskip 1mm
\noindent
{\rm (ii)} ${\beta}$-cocoercive, if there exists a constant $ \beta > 0$, such that
$$
\langle Tx-Ty,x-y \rangle \geq  {\beta}\|Tx-Ty\|^2,\quad\forall x, y\in \mathcal{H}.
$$\vskip 1mm

By the Cauchy--Schwarz inequality, a ${\beta}$-cocoercive operator is $\frac{1}{\beta}$-Lipschitz continuous.
}
\end{dfn}

\begin{lem}{\rm(\cite[Proposition 20.38]{BC2011} \ )}\label{lem3}
Let $A: \mathcal{H}\rightarrow 2^{\mathcal{H}}$ be maximally monotone. Then $\gra(A)$ is sequentially closed in ${\mathcal{H}}^{\rm weak}\times {\mathcal{H}}^{\rm strong}$, i.e., for every sequence $(x^{k},u^{k})_{k \in \mathbb{N}}$ in $\gra(A)$ and $(x,u)\in \mathcal{H}\times \mathcal{H}$, if $x^{k}\rightharpoonup x$ and $u^{k}\rightarrow u$, then $(x,u) \in \gra(A)$. \vskip 1mm
\end{lem}

{
\begin{lem}{\rm(\cite[Proposition 2.3]{Comettes}\ )} \label{Robbins}
Let $F$ be a nonempty closed subset of a separable real Hilbert space $\mathcal{H}$ and  $\phi:[0,+\infty) \rightarrow [0,+\infty)$ be a strictly increasing function such that $\lim_{t \rightarrow +\infty}\phi(t)=+\infty$. Let $\{x^k\}_{k \in \mathbb{N}}$ be a sequence of $\mathcal{H}$-valued random variables and  $\chi_k=\sigma(x^0,\cdots,x^k), \forall k \in \mathbb{N}.$  Suppose that, for every $z \in F$, there exist $\{\beta_k(z)\}_{k\in\mathbb{N}}$, $\{\xi_k(z)\}_{k\in\mathbb{N}}$, and $\{\zeta_k(z)\}_{k\in\mathbb{N}}$ be nonnegative $\chi_k$-measurable random variables such that $\sum_{k=0}^{\infty} \beta_k(z) < \infty$, $\sum_{k=0}^{\infty} \xi_k(z) < \infty$ and
\begin{equation*}
\mathbb{E}(\phi(\|x^{k+1}-z\|) |  \chi_k) \leq (1+\beta_k(z))\phi(\|x^k-z\|)+\xi_k(z)-\zeta_k(z),\quad\forall k\in\mathbb{N}.
\end{equation*}
Then the following hold:\vskip 1mm
\noindent
{\rm (i)}  $(\forall z \in F),$ $\sum_{k=0}^{\infty} \zeta_k(z) < \infty$ almost surely. \vskip 1mm
\noindent
{\rm (ii)} There exists $\Xi \in \mathcal{F}$ such that $P(\Xi)=1$, for every $\theta \in \Xi$ and every $z \in F$, $\{\|x^k(\theta)-z\|\}_{k\in \mathbb{N}}$ converges.  \vskip 1mm
\noindent
{\rm (iii)} Suppose that all weak cluster points of $\{x^k\}_{k \in \mathbb{N}}$ belong to $F$ almost surely, then $\{x^k\}_{k \in \mathbb{N}}$ converges weakly almost surely to an $F$-valued random variable.
\end{lem}
}
The paper is organized as follows. In Section 2, we introduce the stochastic forward-backward-half forward splitting algorithm with variance reduction to solve the problem \eqref{PROB2}, and show the weak almost sure and linear convergence of the proposed algorithm. Finally, we present the numerical experiments in Section 3.
\section{Main Results}\label{main}
In the sequel, we assume that the following conditions are satisfied:
\begin{asm} \label{CON1}
	\begin{itemize}
		\item[(i)] The operator $A: \mathcal{H}\to 2^{\mathcal{H}}$ is maximal monotone;
        \item[(ii)] The operator $B$ has a stochastic oracle $B_\xi$ that is unbiased, $B(x)=\mathbb{E}[B_{\xi}(x)]$, and $L$-Lipschitz in mean:
        \begin{equation}
        \label{lip}
        \mathbb{E}[\|B_{\xi}(u)-B_{\xi}(v)\|^2] \leq L^2\|u-v\|^2, \quad \forall u,v \in \mathcal{H};
        \end{equation}
		\item[(iii)] $C: \mathcal{H}\to \mathcal{H}$ is $\beta$-cocoercive;
		\item[(iv)] The solution set of the problem \eqref{PROB2}, denoted by {$Z$}, is nonempty.
	\end{itemize}
\end{asm}

\noindent
We now present the stochastic forward-backward-half forward splitting algorithm with variance reduction to solve the  problem \eqref{PROB2}.

\begin{alg}\label{ALG}
{\rm
\hrule\hrule
\vskip 1mm
\noindent\textbf{VRFBHF}
\vskip 1mm
\hrule\hrule

\vskip 1mm

\noindent
1. {\bf Input:}
Probability $p \in (0, 1]$, probability distribution $Q$, step-size $\gamma$, $\lambda \in (0, 1).$

Let $x^0 = w^0$.

\noindent
2. {\bf for} $k = 0, 1, \ldots$ {\bf do}

\noindent
3.\quad  $\bar x^k=\lambda x^k+(1-\lambda)w^k$

\noindent
4.\quad
$
y^k = J_{\gamma A}\left(\bar x^k - \gamma(B+C)w^k\right)
$

\noindent
5.\quad
Draw an index $\xi_k$ according to $Q$

\noindent
6.\quad
$
x^{k+1}=y^k+\gamma(B_{\xi_k}w^k-B_{\xi_k}y^k)
$

\noindent
7.\quad
$
w^{k+1}=
\begin{cases}
x^{k+1},&\hbox{with probability} \,\,p\\
w^k,&\hbox{with probability} \,\,1-p\
\end{cases}
$

\noindent
8:\quad \textbf{end \ for}
\vskip 1mm
}
\hrule\hrule
\hspace*{\fill}
\end{alg}

\begin{rem}
\rm

Algorithm \ref{ALG} is a very general algorithm and it is brand new to the literature. We review how Algorithm \ref{ALG} relates to previous work. Algorithm \ref{ALG} becomes the forward-backward-forward algorithm with variance reduction in \cite{Alacaoglu-Malitsky} if $C=0$. Algorithm \ref{ALG} reduces to loopless SVRG in \cite{KHR} if $\lambda=1$,
$B=\nabla f$, $A=0$ and $C=0$, where $f(x)=\sum_{i=1}^Nf_i(x)$ and $f_i(x)$ is the loss of model $x$ on data point $i$.

\end{rem}
\begin{rem}
We have two sources of randomness at each iteraton: the index $\xi_k$ which is used for updating $x^{k+1}$, and the reference point $w^k$ which is updated in each iteration with probability $p$ by the iterate $x^{k+1}$, or left unchanged with probability $1-p$. Intuitively, we wish to keep $p$ small to lower the cost per iteration. And different from the FBHF splitting algorithm \eqref{FBHF},  we use the parameter $\lambda $ to introduce inertia in  Algorithm \ref{ALG} by including the information of past iterations. This can improve the efficiency of the algorithms. See \cite{Nesterov1,{Nesterov2}} for details.
\end{rem}
\subsection{The weak almost sure convergence}

In this subsection, we establish the weak almost sure convergence of Algorithm \ref{ALG}. We use the following notations for conditional expectations: $\mathbb{E}_k[\cdot]=\mathbb{E}[\cdot|\sigma(\xi_0,...,\xi_{k-1},w^{k})]$ and $\mathbb{E}_{k+\frac{1}{2}}[\cdot]=\mathbb{E}[\cdot|\sigma(\xi_0,...,\xi_{k},w^{k})]$.

For the iterates $\{x^k\}_{k \in \mathbb{N}}$ and $\{w^k\}_{k \in \mathbb{N}}$ generated by Algorithm \ref{ALG}, we define the Lyapunov function
\begin{equation*}
\Phi_{k}(x):=\lambda \|x^k-x\|^2+\frac{1-\lambda}{p}\|w^k-x\|^2,\ \forall x \in \mathcal{H},
\end{equation*}
 which helps to establish the weak almost sure convergence of the proposed algorithm.
\begin{theorem}
\label{the1}
{
\noindent
Let Assumption \ref{CON1} hold, $\lambda \in [0,1)$, $p \in (0,1]$, and $\gamma \in (0,\frac{4\beta(1-\lambda)}{1+\sqrt{1+16\beta^2L^2(1-\lambda)}})$. Then for $\{x^k\}_{k \in \mathbb{N}}$ generated by Algorithm \ref{ALG} and any $x^* \in Z$, it holds that
\begin{equation}
\label{lypunov}
\mathbb{E}_k[\Phi_{k+1}(x^*)] \leq \Phi_{k}(x^*).
\end{equation}
Then the sequence $\{x^k\}_{k \in \mathbb{N}}$  generated by Algorithm \ref{ALG} {converges weakly almost surely to a $Z$-valued random variable}.
}
\end{theorem}
\begin{proof}
Since $x^* \in \zer(A+B+C),$ we have
\begin{equation}
-\gamma(B+C)x^* \in \gamma Ax^*. \label{1}
\end{equation}
Step 4 in Algorithm \ref{ALG} is equivalent to the inclusion
\begin{equation}
\bar{x}^k-y^k-\gamma(B+C)w^k \in \gamma Ay^k. \label{2}
\end{equation}
Combining \eqref{1}, \eqref{2} and the monotonicity of $A$, we have
\begin{equation*}
\langle y^k-\bar{x}^k+\gamma(B+C)w^k , x^*-y^k \rangle-\gamma \langle(B+C)x^*, x^*-y^k \rangle \geq 0.
\label{3}
\end{equation*}
Then from step 6 in Algorithm \ref{ALG}, we obtain
\begin{equation}
\langle x^{k+1}-\bar{x}^k+\gamma(Bw^k-B_{\xi_k}w^k+B_{\xi_k}y^k)+\gamma Cw^k , x^*-y^k \rangle-\gamma \langle(B+C)x^*, x^*-y^k \rangle \geq 0.
\label{4}
\end{equation}
By the definition of $\bar{x}^k$ and identities $2\langle a , b \rangle=\|a+b\|^2-\|a\|^2-\|b\|^2=\|a\|^2+\|b\|^2-\|a-b\|^2$, we have
\begin{equation}
\aligned
&2\langle x^{k+1}-\bar{x}^k, x^*-y^k \rangle\\
=& 2\langle x^{k+1}-y^k, x^*-y^k \rangle + 2\langle y^k-\bar{x}^k, x^*-y^k \rangle \\
=& \|x^{k+1}-y^k\|^2+\|x^*-y^k\|^2-\|x^{k+1}-x^*\|^2
 + 2\lambda\langle y^k-x^k , x^*-y^k \rangle \\
 &+ 2(1-\lambda)\langle y^k-w^k , x^*-y^k \rangle \\
=& \|x^{k+1}-y^k\|^2+\|x^*-y^k\|^2-\|x^{k+1}-x^*\|^2\\
&+ \lambda(\|x^k-x^*\|^2-\|y^k-x^k\|^2-\|y^k-x^*\|^2)\\
& +(1-\lambda)(\|w^k-x^*\|^2-\|y^k-w^k\|^2-\|y^k-x^*\|^2)\\
=&\|x^{k+1}-y^k\|^2-\|x^{k+1}-x^*\|^2+ \lambda\|x^k-x^*\|^2-\lambda\|y^k-x^k\|^2 \\
& +(1-\lambda)\|w^k-x^*\|^2-(1-\lambda)\|y^k-w^k\|^2.
\endaligned
\label{5}
\end{equation}
By the $\beta$-cocoercivity of $C$ and Young's inequality $\langle a,b \rangle \leq \beta\|a\|^2 + \frac{1}{4\beta} \|b\|^2$ for all $a , b \in \mathcal{H},$  we get
\begin{equation}
\label{7}
\aligned
&2\gamma \langle Cw^k-Cx^* , x^*-y^k \rangle\\
=&2\gamma \langle Cw^k-Cx^* , x^*-w^k \rangle+2\gamma \langle Cw^k-Cx^* , w^k-y^k \rangle \\
\leq& -2\gamma\beta\|Cw^k-Cx^*\|^2+2\gamma\beta\|Cw^k-Cx^*\|^2+\frac{\gamma}{2\beta}\|w^k-y^k\|^2 \\
=&\frac{\gamma}{2\beta}\|w^k-y^k\|^2.
\endaligned
\end{equation}
We use \eqref{5} and \eqref{7} in \eqref{4} to obtain
\begin{equation}
\label{8-}
\aligned & \ \ \ \ 2\gamma \langle Bx^*-(Bw^k-B_{\xi_k}w^k+B_{\xi_k}y^k) , x^*-y^k \rangle + \|x^{k+1}-x^*\|^2 \\
& \leq \lambda\|x^k-x^*\|^2+(1-\lambda)\|w^k-x^*\|^2+\|x^{k+1}-y^k\|^2 \\
&\ \ \ \ -\lambda\|y^k-x^k\|^2-(1-\lambda-\frac{\gamma}{2\beta})\|y^k-w^k\|^2.
\endaligned
\end{equation}
Taking expectation $\mathbb{E}_k$ on \eqref{8-} and using
\begin{equation*}
\mathbb{E}_k [ \langle Bw^k-B_{\xi_k}w^k+B_{\xi_k}y^k, x^*-y^k \rangle]=\langle By^k, x^*-y^k \rangle ,
\end{equation*}
we obtain
\begin{equation*}
\label{8}
\aligned & \ \ \ \ 2\gamma \langle Bx^*-By^k , x^*-y^k \rangle + \mathbb{E}_k \|x^{k+1}-x^*\|^2 \\
& \leq \lambda\|x^k-x^*\|^2+(1-\lambda)\|w^k-x^*\|^2+\mathbb{E}_k \|x^{k+1}-y^k\|^2 \\
&\ \ \ \ -\lambda\|y^k-x^k\|^2-(1-\lambda-\frac{\gamma}{2\beta})\|y^k-w^k\|^2.
\endaligned
\end{equation*}
By the monotonicity of $B,$ we have
\begin{equation}
\label{9}
\langle Bx^*-By^k , x^*-y^k \rangle \geq 0.
\end{equation}
Combining the definition of $x^{k+1}$ and \eqref{lip}, we have
\begin{equation*}
\mathbb{E}_k \|x^{k+1}-y^k\|^2 \leq \gamma^2L^2\|y^k-w^k\|^2.
\end{equation*}
Therefore,
\begin{equation}
\label{8+}
\aligned \mathbb{E}_k \|x^{k+1}-x^*\|^2
& \leq \lambda\|x^k-x^*\|^2+(1-\lambda)\|w^k-x^*\|^2-\lambda\|y^k-x^k\|^2 \\
&\ \ \ \ -(1-\lambda-\gamma^2L^2-\frac{\gamma}{2\beta})\|y^k-w^k\|^2.
\endaligned
\end{equation}
On the other hand, the definition of $w^{k+1}$ and $\mathbb{E}_{k+\frac{1}{2}}$ yield that
\begin{equation}
\label{e2+1}
\frac{1-\lambda}{p}\mathbb{E}_{k+\frac{1}{2}}[\|w^{k+1}-x^*\|^2]=(1-\lambda)\|x^{k+1}-x^*\|^2+(1-\lambda)\frac{1-p}{p}\|w^k-x^*\|^2.
\end{equation}
Then apply to \eqref{e2+1} the tower property $\mathbb{E}_k[\mathbb{E}_{k+\frac{1}{2}}[\cdot]]=\mathbb{E}_k[\cdot]$, we have
\begin{equation}
\label{e2+2}
\frac{1-\lambda}{p}\mathbb{E}_{k}[\|w^{k+1}-x^*\|^2]=(1-\lambda)\mathbb{E}_{k}\|x^{k+1}-x^*\|^2+(1-\lambda)\frac{1-p}{p}\|w^k-x^*\|^2.
\end{equation}
We add \eqref{e2+2} to \eqref{8+} to obtain
\begin{equation}
\label{10}
\aligned
\mathbb{E}_k [\Phi_{k+1}(x^*)]
\leq \Phi_{k}(x^*)-\lambda\|y^k-x^k\|^2  -(1-\lambda-\gamma^2L^2-\frac{\gamma}{2\beta})\|y^k-w^k\|^2.
\endaligned
\end{equation}
Thus, the inequality \eqref{lypunov} holds with $\gamma \in (0,\frac{4\beta(1-\lambda)}{1+\sqrt{1+16\beta^2L^2(1-\lambda)}})$ and $0<\lambda<1$.

Next, we show the weak almost sure convergence of the sequence $\{x^{k}\}_{k \in \mathbb{N}}$. {By Lemma \ref{Robbins} (i), there exists $\Xi\in \mathcal{F}$ such that $\mathbb{P}(\Xi)=1$ and $ \forall \theta \in \Xi$, $y^k(\theta)-x^k(\theta) \rightarrow 0$, $y^k(\theta)-w^k(\theta) \rightarrow 0$ as $k\rightarrow \infty$,  which implies $y^k(\theta)-\bar{x}^k(\theta) \rightarrow 0$ as $k\rightarrow \infty$.
From  Lemma \ref{Robbins} (ii)  there exists $\Xi^{'} \in \mathcal{F}$  such that $\mathbb{P}(\Xi^{'})=1$ and
$\{\lambda \|x^k(\theta)-x^*\|^2+\frac{1-\lambda}{p}\|w^k(\theta)-x^*\|^2\}_{k\in \mathbb{N}} \ \hbox{converges}$ for $ \forall  \theta \in \Xi^{'},$ $\forall x^* \in Z$,  which yields that the sequence $\{x^{k}(\theta)\}_{k \in \mathbb{N}}$ is bounded. } Pick $\theta \in \Xi \bigcap \Xi^{'}$ and let $\{x^{k_j}(\theta)\}_{j \in \mathbb{N}}$ be a weak convergent subsequence of the sequence $\{x^{k}(\theta)\}_{k \in \mathbb{N}}$, say without loss of generality that $x^{k_j}(\theta) \rightharpoonup \bar{x}(\theta)$ as $j\rightarrow \infty$.  From  $y^{k_j}(\theta)-x^{k_j}(\theta) \rightarrow 0$ as $j\rightarrow \infty$, it follows that $y^{k_j}(\theta) \rightharpoonup \bar{x}(\theta)$ as $j\rightarrow \infty$.
Then according to \eqref{2}, we can get
\begin{equation*}
\bar{x}^{k_j}(\theta)-y^{k_j}(\theta)-\gamma((B+C)w^{k_j}(\theta)-(B+C)y^{k_j}(\theta))\in \gamma(A+B+C)y^{k_j}(\theta).
\end{equation*}
Using the Lipschitz property of $B+C$, we get
\begin{equation*}
\bar{x}^{k_j}(\theta)-y^{k_j}(\theta)-\gamma((B+C)w^{k_j}(\theta)-(B+C)y^{k_j}(\theta)) \rightarrow 0,\ \hbox{as}\ j\rightarrow \infty.
\end{equation*}
Furthermore, based on the assumption that the operator $B$ has a full domain, we have that $A+B$ is maximally monotone by Corollary 25.5 (i) in \cite{BC2011}. Combining Lemma 2.1 in \cite{Showalter} and the assumption that $C$ is cocoercive, one has that $A+B+C$ is maximally monotone. By Lemma \ref{lem3}, $(\bar{x}(\theta) , 0) \in \gra(A+B+C)$, i.e., {$\bar{x}(\theta) \in Z$}. Hence, all weak cluster points of $\{x^{k}(\theta)\}_{k \in \mathbb{N}}$ and $\{w^{k}(\theta)\}_{k \in \mathbb{N}}$ belong to {$Z$}. {By Lemma \ref{Robbins} (iii), the sequence $\{x^k\}_{k \in \mathbb{N}}$ converges weakly almost surely to a $Z$-valued random variable.}
\end{proof}

\subsection{Linear convergence}
In this subsection, we show the linear convergence of Algorithm \ref{ALG} for solving the structured monotone inclusion problem \eqref{PROB2} when $B$ is $\mu$-strongly monotone. Indeed, assuming that the operator $A$ is strongly monotone also leads to a linear convergence result, and the proof procedure is similar.
\begin{theorem}
\label{theorem2}
{
\noindent
Let Assumption \ref{CON1} hold, $B$ be $\mu$-strongly monotone and {$x^*$ be the solution of the problem \eqref{PROB2}}. If we set $\lambda=1-p$, and $\gamma=\min\{\frac{\sqrt{p}}{2L},\beta p\}$ in Algorithm \ref{ALG}, then for the sequence $\{x^k\}_{k \in \mathbb{N}}$ generated by Algorithm \ref{ALG}, it holds that
\begin{equation}
\label{xx}
\mathbb{E}\|x^k-x^*\|^2 \leq (\frac{1}{1+c/4})^k \frac{2}{1-p} \|x^0-x^*\|^2,
\end{equation}
with $c=\min \{ \gamma \mu, \frac{p}{(1+\sqrt p)(4+p)}\}$.
}
\end{theorem}
\begin{proof}
If $B$ is $\mu$-strongly monotone, then \eqref{9} becomes
\begin{equation*}
\langle Bx^*-By^k , x^*-y^k \rangle \geq \mu\|x^*-y^k\|^2.
\end{equation*}
We continue as in the proof of Theorem \ref{the1} to obtain, instead of \eqref{10},
\begin{equation}
\label{11}
\aligned
& \ \ \ \ 2\gamma\mu\|y^k-x^*\|^2+\lambda \mathbb{E}_k \|x^{k+1}-x^*\|^2+\frac{1-\lambda}{p} \mathbb{E}_k\|w^{k+1}-x^*\|^2 \\
& \leq \lambda \|x^k-x^*\|^2+\frac{1-\lambda}{p}\|w^k-x^*\|^2-\lambda\|y^k-x^k\|^2 \\
& \ \ \ \ -(1-\lambda-\gamma^2L^2-\frac{\gamma}{2\beta})\|y^k-w^k\|^2.
\endaligned
\end{equation}
By $\|a+b\|^2 \leq 2\|a\|^2 + 2\|b\|^2$, the step 6 and \eqref{lip}, we have
\begin{equation}
\label{12}
\aligned
2\gamma\mu\|y^k-x^*\|^2
& \geq \gamma \mu \mathbb{E}_k[\|x^{k+1}-x^*\|^2]-2\gamma\mu \mathbb{E}_k[\|\gamma(B_{\xi_k}w^k - B_{\xi_k}y^k)\|^2] \\
& \geq \gamma \mu \mathbb{E}_k[\|x^{k+1}-x^*\|^2]-2\gamma^3L^2\mu\|y^k-w^k\|^2.
\endaligned
\end{equation}
Combining \eqref{11}, \eqref{12} and $\lambda=1-p$,  we get
\begin{equation}
\label{14}
\aligned
& \ \ \ \ (1-p+\gamma \mu)\mathbb{E}_k [\|x^{k+1}-x^*\|^2]+\mathbb{E}_k[\|w^{k+1}-x^*\|^2] \\
& \leq (1-p)\|x^k-x^*\|^2+\|w^k-x^*\|^2-(1-p)\|y^k-x^k\|^2 \\
& \ \ \ \ -(p-\gamma^2L^2-\frac{\gamma}{2\beta}-2\gamma^3L^2\mu )\|y^k-w^k\|^2 \\
& \leq (1-p)\|x^k-x^*\|^2+\|w^k-x^*\|^2-(1-p)\|y^k-x^k\|^2 \\
& \ \ \ \ -\frac{p(1-\sqrt{p})}{4}\|y^k-w^k\|^2,
\endaligned
\end{equation}
where the last inequality is obtained by $\gamma=\min\{\frac{\sqrt{p}}{2L},\beta p\}$ and $\mu \leq L$. Similar to \eqref{12}, we have
\begin{equation}
\label{13}
\aligned
&\frac{c}{2}\mathbb{E}_k[\|x^{k+1}-x^*\|^2]\\
 \geq& \frac{c}{4} \mathbb{E}_k[\|w^{k+1}-x^*\|^2] - \frac{c}{2}\mathbb{E}_k[\mathbb{E}_{k+\frac{1}{2}}\|x^{k+1}-w^{k+1}\|^2] \\
 =& \frac{c}{4} \mathbb{E}_k[\|w^{k+1}-x^*\|^2] - \frac{c(1-p)}{2}\mathbb{E}_k[\|x^{k+1}-w^k\|^2] \\
 =& \frac{c}{4} \mathbb{E}_k[\|w^{k+1}-x^*\|^2] - \frac{c(1-p)}{2}\mathbb{E}_k[\|y^k-w^k+\gamma(B_{\xi_k}w^k-B_{\xi_k}y^k)\|^2] \\
 \geq& \frac{c}{4} \mathbb{E}_k[\|w^{k+1}-x^*\|^2] -c(1-p)(1+\gamma^2L^2)\|y^k-w^k\|^2 \\
 \geq& \frac{c}{4} \mathbb{E}_k[\|w^{k+1}-x^*\|^2] - \frac{c(1-p)(4+p)}{4}\|y^k-w^k\|^2.
\endaligned
\end{equation}
Putting \eqref{13} into \eqref{14} and recalling that $c \leq \gamma \mu$,  we have
\begin{equation}
\label{15}
\aligned
& \ \ \ \ (1-p+\frac{c}{2})\mathbb{E}_k [\|x^{k+1}-x^*\|^2]+(1+\frac{c}{4})\mathbb{E}_k[\|w^{k+1}-x^*\|^2] \\
& \leq (1-p)\|x^k-x^*\|^2+\|w^k-x^*\|^2-(1-p)\|y^k-x^k\|^2 \\
& \ \ \ \ -\left[\frac{p(1-\sqrt{p})}{4}-\frac{c(1-p)(4+p)}{4}\right]\|y^k-w^k\|^2\\
& \leq(1-p)\|x^k-x^*\|^2+\|w^k-x^*\|^2,
\endaligned
\end{equation}
where the last inequality comes from $c \leq \frac{p}{(1+\sqrt p)(4+p)}.$
Then, using $1-p+\frac{c}{2} \geq (1-p)(1+\frac{c}{4})$ and taking the full expectation on \eqref{15}, we have
\begin{equation*}
(1+\frac{c}{4})\mathbb{E}[(1-p)\|x^{k+1}-x^*\|^2+\|w^{k+1}-x^*\|^2]  \leq \mathbb{E}[(1-p)\|x^k-x^*\|^2+\|w^k-x^*\|^2].
\end{equation*}
Iterating this inequality, we obain
\begin{equation*}
(1-p)\mathbb{E}\|x^{k}-x^*\|^2  \leq(\frac{1}{1+c/4})^k (2-p)\|x^0-x^*\|^2,
\end{equation*}
showing \eqref{xx}.
\end{proof}

\section{Numerical Simulations}\label{numerics}

\noindent In this section, we compare the Algorithm \ref{ALG} (VRFBHF) with the FBHF splitting algorithm \eqref{FBHF}. Consider the nonlinear constrained optimization problem of the form
\begin{equation}
\label{fh}
\min_{x\in C}f(x)+h(x),
\end{equation}
where $C= \{ x \in \mathcal{H} \,|\, ( \forall i \in \{1,...,q\}) \ g_i(x) \leq 0\} $, $f: \mathcal{H} \rightarrow ( -\infty ,+\infty]$ is a proper, convex  and  lower semi-continuous function, for every $i \in \{1,...,q\}$, $g_i : \dom(g_i) \subset \mathcal{H} \rightarrow \mathbb{R}$ and $h: \mathcal{H} \rightarrow \mathbb{R}$ are $C^1$ convex functions in $\nt \dom g_i$ and $\mathcal{H}$, respectively, and $\nabla h$ is $\beta$-Lipschitz.  The solution to the optimization problem \eqref{fh} can be found via the saddle points of the Lagrangian
\begin{equation*}
L(x,u)=f(x)+h(x)+u^{\top}g(x)-\iota_{\mathbb{R}_+^q}(u),
\end{equation*}
where $\iota_{\mathbb{R}_+^q}$ is the indicator function of $\mathbb{R}_+^q$, Under some standard qualifications, the solution to the optimization problem \eqref{fh} can be found by solving the monotone inclusion \cite{{FBHF},RT}: find $x \in Y$ such that $\exists u \in \mathbb{R}_{+}^q$,
\begin{equation}
\label{ABCxu}
(0,0) \in (A+B+C)(x,u),
\end{equation}
where $Y \subset \mathcal{H}$ is a nonempty closed convex set modeling the prior information of the solution, $A:(x,u)\mapsto\partial f(x)\times{N_{\mathbb{R}_+^q}u}$ is a maximally monotone, $C:(x,u)\mapsto(\nabla h(x),0)$ is $\beta$-cocoercive, and $ B:(x,u)\mapsto(\sum_{i=1}^q u_i \nabla g_i(x),-g_1(x),...,-g_q(x))$ is nonlinear, monotone and continuous.
\vskip 2mm
\begin{exm} \label{ex1}
\rm
\noindent Let $\mathcal{H}=\mathbb{R}^d,$ $f=\iota_{[0,1]^d}$, {$g_i(x)=s_i^{\top}x$ ($\forall i \in  \{1,...,q\}$) with $s_1,\ldots,s_q\in \mathbb{R}^d$,} and  $h=\frac{1}{2}\|Gx-b\|^2$ with $G$ being an $t \times d$ real matrix, $d=2t$, $b\in \mathbb{R}^t$. Then the  operators in \eqref{ABCxu} become
\begin{equation}
\label{A1}
\aligned
&A:(x,u)\mapsto\partial{\iota_{[0,1]^d}(x)}\times{N_{\mathbb{R}_+^q}u},\\
& B:(x,u)\mapsto(D^{\top}u,-Dx),\\
& C:(x,u)\mapsto(G^{\top}(Gx-b),0),\\
 \endaligned
\end{equation}
where  $x\in \mathbb{R}^d$, $u\in \mathbb{R}_+^q$, {$D=[s_1,\ldots,s_q]^{\top}$.} It is easy to see that the operator $A$ is a maximally monotone operator, $C$ is a $\beta$-cocoercive operator with $\beta =\|G\|^{-2}$, $B$ is a $L_B$-Lipschitz operator with $L_B=\|D\|$.
In the light of the structure of the operator $B$, rewrite $B$ as $B= \sum_{i=1}^{q+d} B_i.$ For uniform sampling, the stochastic oracle $B_{\xi}(x,u)=(q+d)B_i(x,u)$, $P_{\xi}(i)=\Pro{ \{\xi=i\}}=\frac{1}{q+d}$, $i \in \{1,...,q+d\}$.

\begin{figure}[h]
	
	\begin{minipage}{0.32\linewidth}
		\vspace{3pt}
		\centerline{\includegraphics[width=\textwidth]{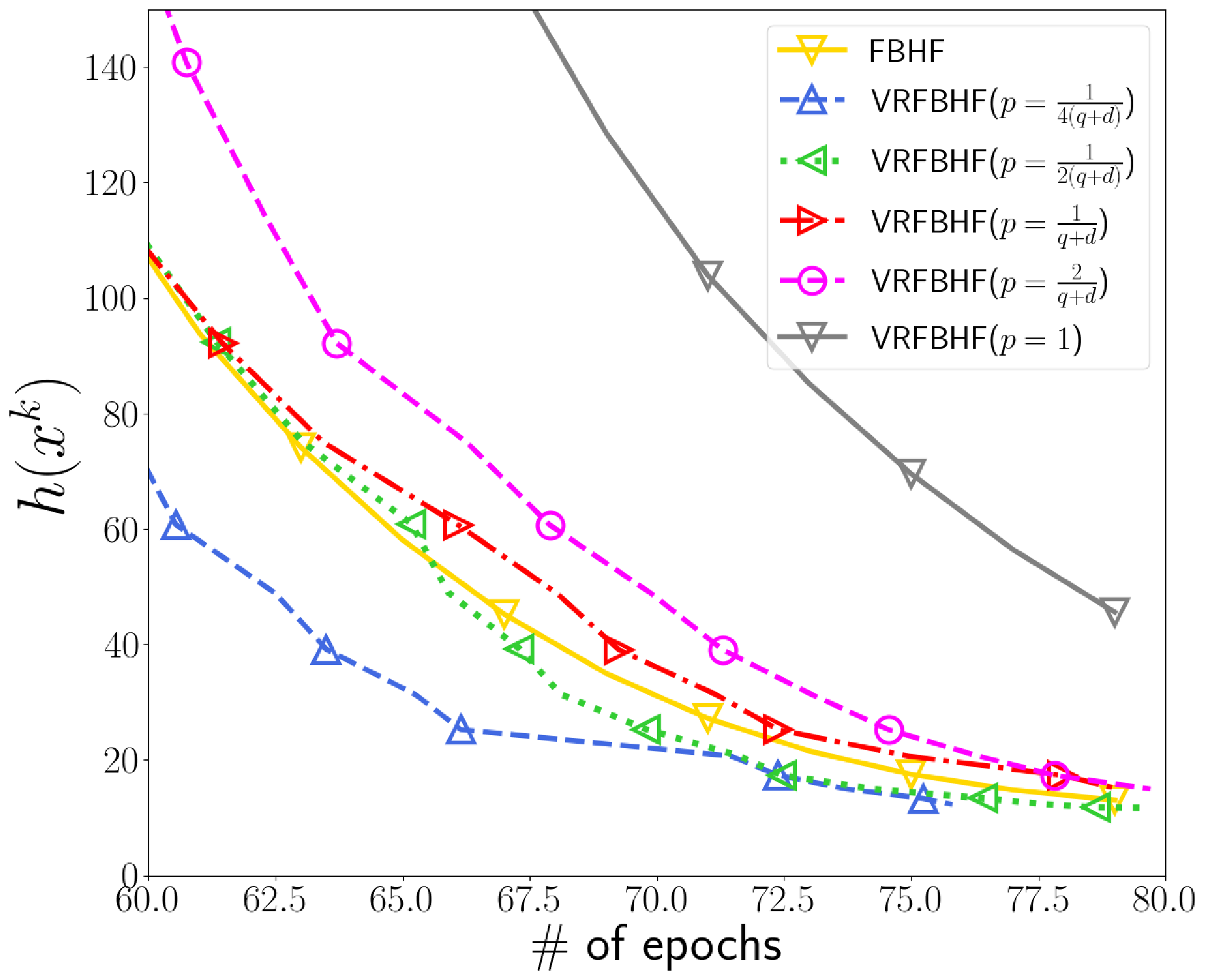}}
		\centerline{$q=1000,d=500$}
	\end{minipage}
	\begin{minipage}{0.32\linewidth}
		\vspace{3pt}
		\centerline{\includegraphics[width=\textwidth]{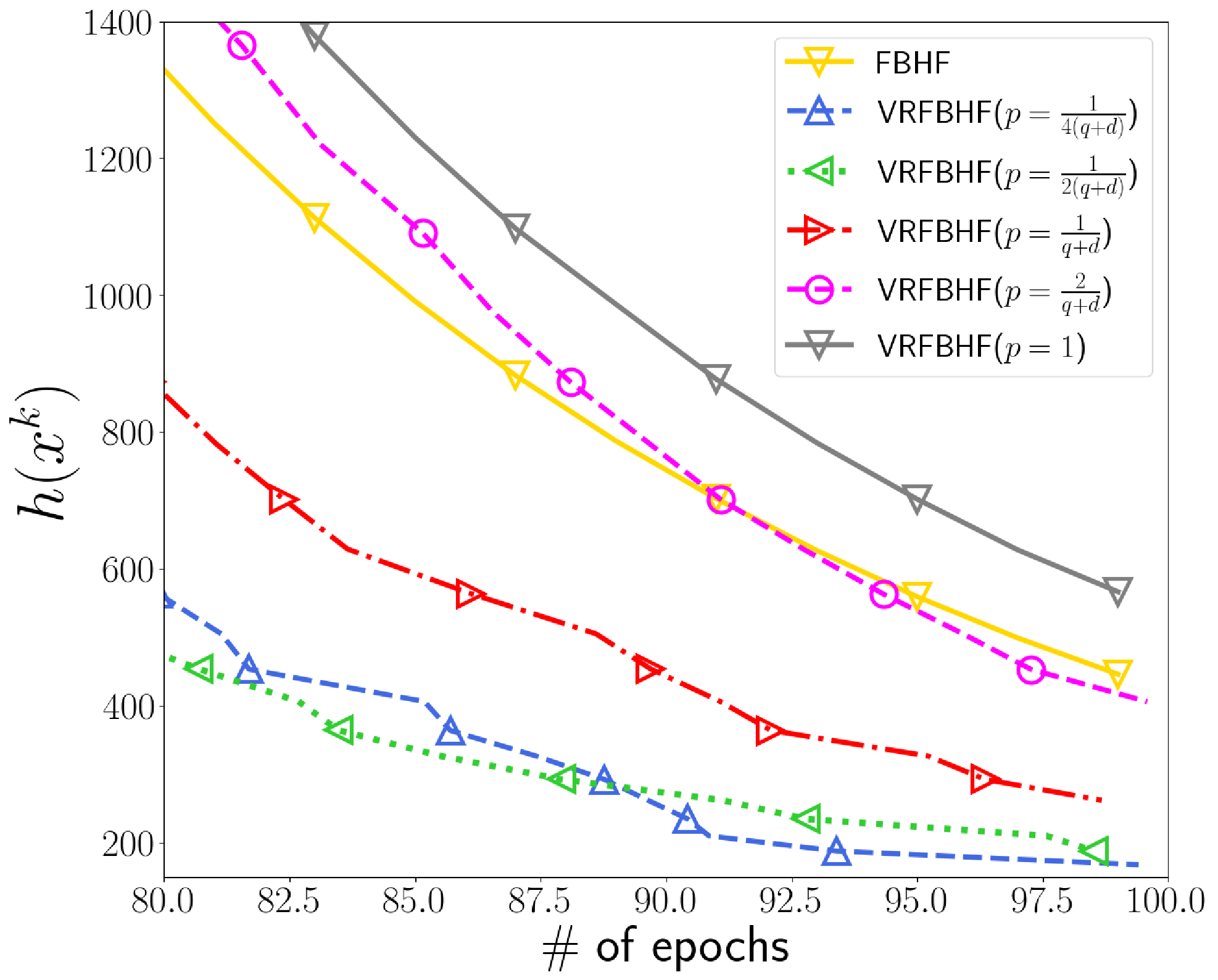}}	
		\centerline{$q=1000,d=1000$}
	\end{minipage}
	\begin{minipage}{0.32\linewidth}
		\vspace{3pt}
		\centerline{\includegraphics[width=\textwidth]{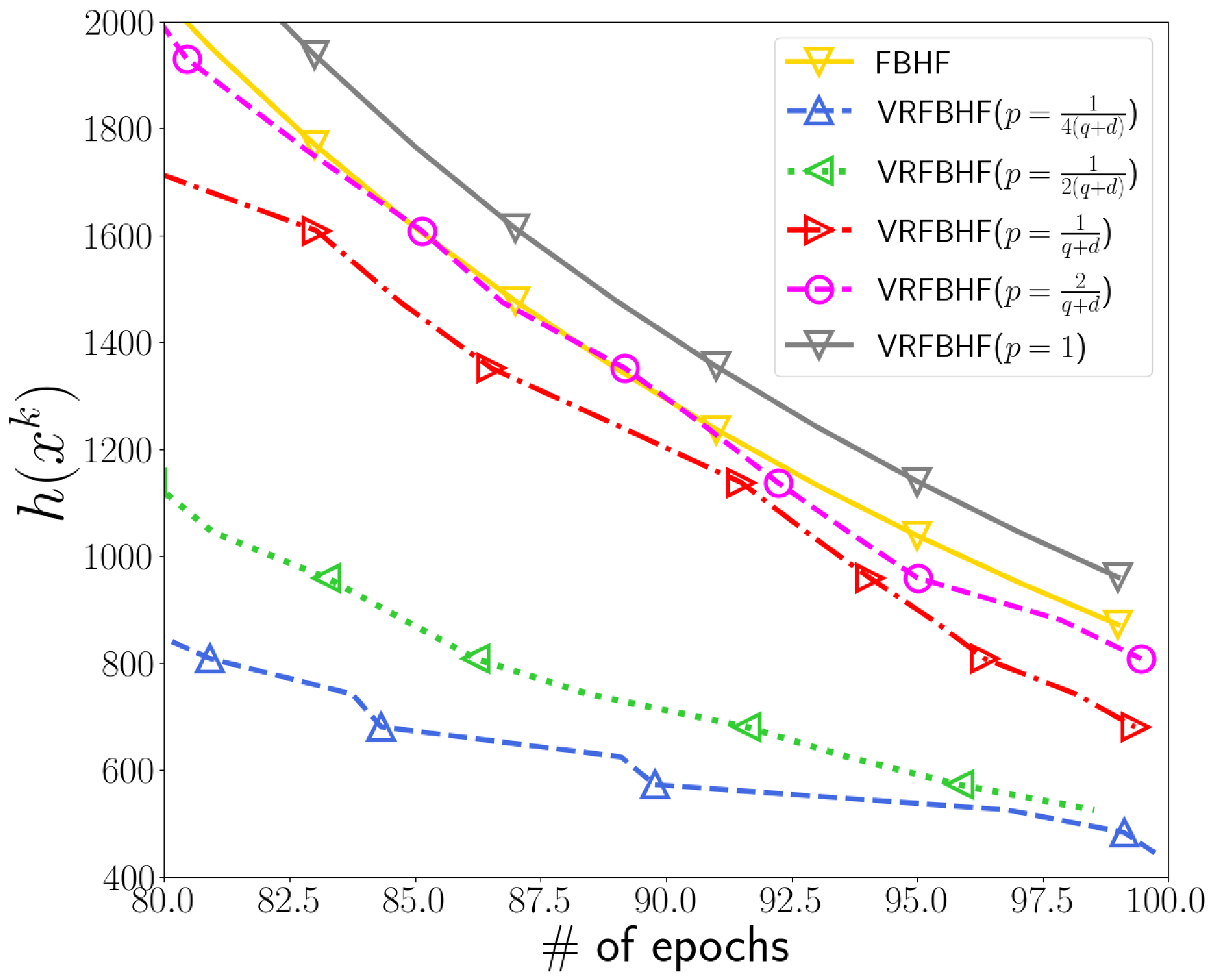}}	
		\centerline{$q=500,d=1000$}
	\end{minipage}

	\caption{Decay of $h(x^k)$ with the number of epochs }
	\label{fig1}
\end{figure}

\begin{figure}[h]
	
	\begin{minipage}{0.32\linewidth}
		\vspace{3pt}
		\centerline{\includegraphics[width=\textwidth]{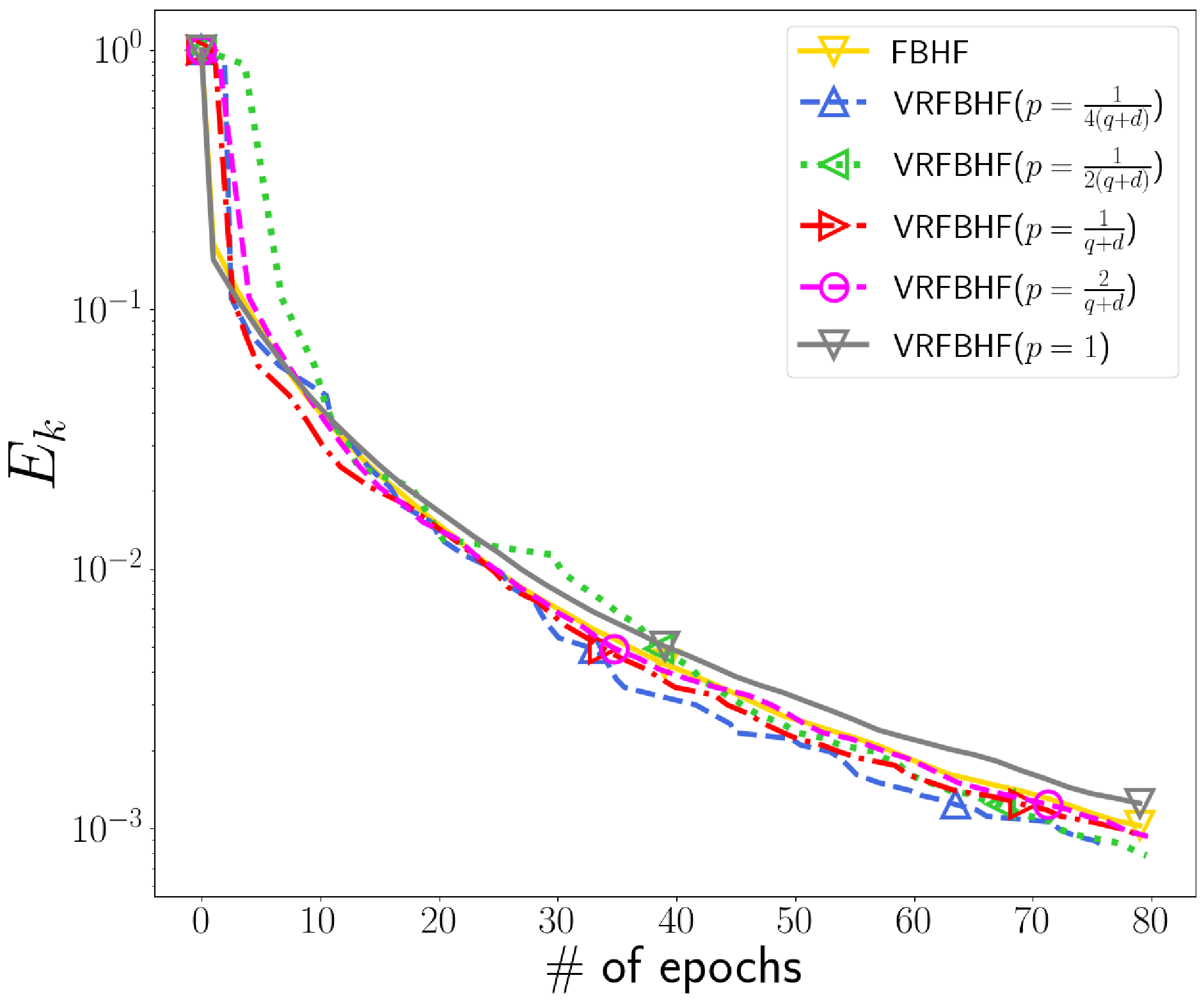}}
		\centerline{$q=1000,d=500$}
	\end{minipage}
	\begin{minipage}{0.32\linewidth}
		\vspace{3pt}
		\centerline{\includegraphics[width=\textwidth]{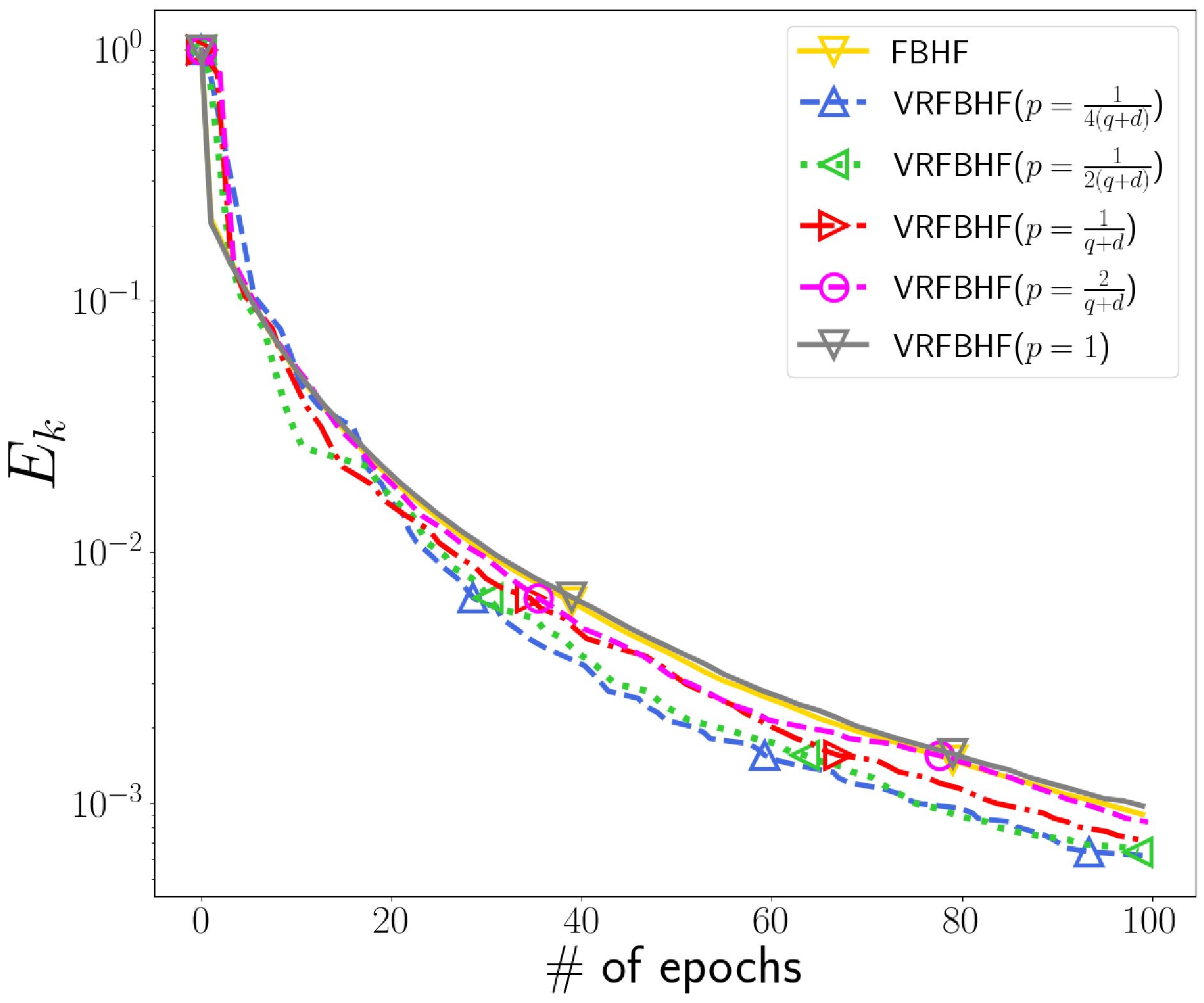}}	
		\centerline{$q=1000,d=1000$}
	\end{minipage}
	\begin{minipage}{0.32\linewidth}
		\vspace{3pt}
		\centerline{\includegraphics[width=\textwidth]{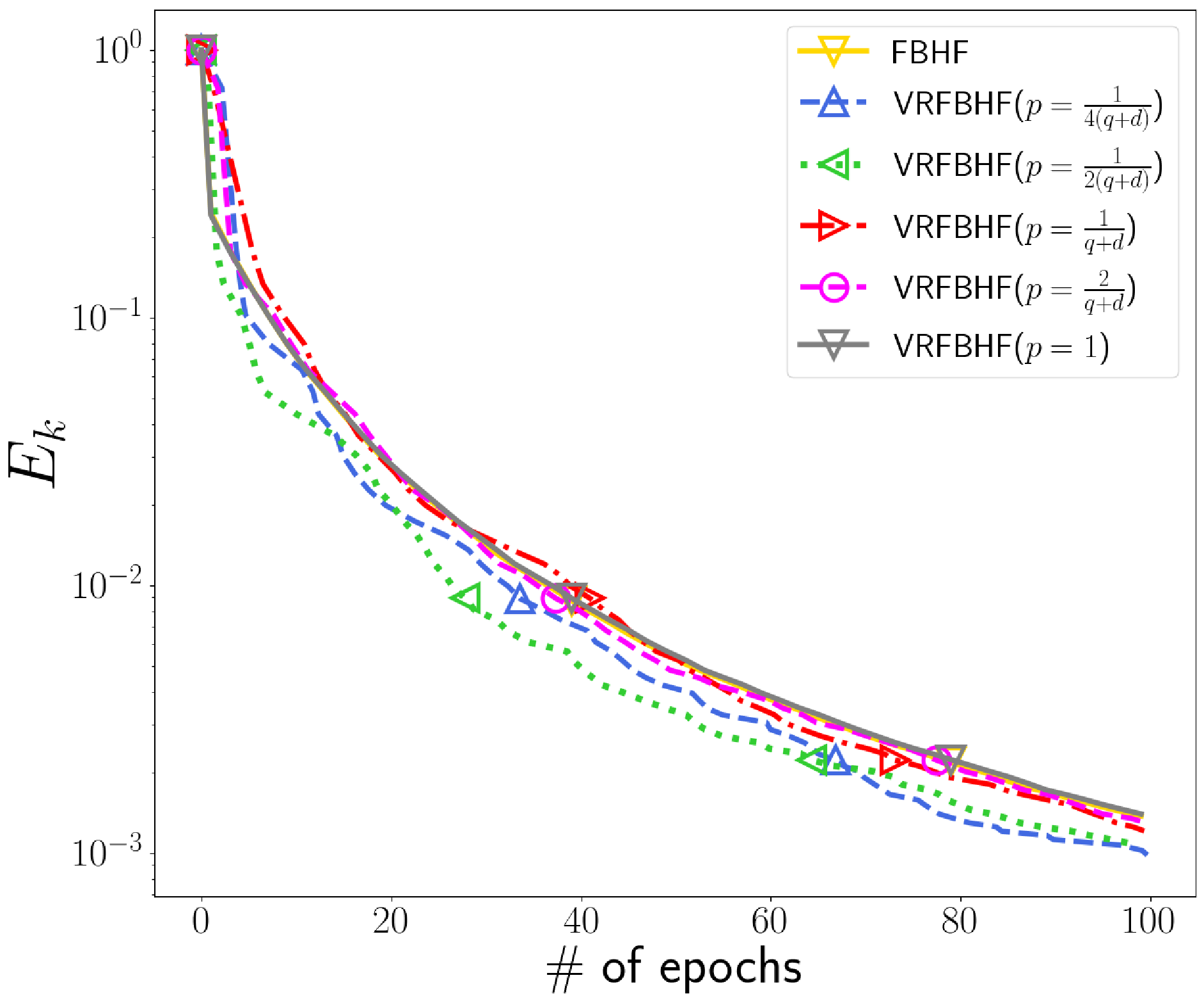}}	
		\centerline{$q=500,d=1000$}
	\end{minipage}

	\caption{Decay of $E_k$ with the number of epochs }
	\label{fig2}
\end{figure}

In the numerical test,  $G,D,b$ and initial value $(x_{0},u_{0})$  are all randomly generated. In VRFBHF, set  $(w_{0},v_{0})=(x_{0},u_{0}),$ take $\lambda=0.1,$ and $\gamma=\frac{\beta(1-\lambda)}{1+\sqrt{1+16 \beta^2L^2(1-\lambda)}}$. In FBHF, take $\gamma =\frac{\beta}{1+\sqrt{1+16\beta^2L_B^2}}.$ We test three problem sizes, it is observed from Figure \ref{fig1} that {$h(x^k)$} of VRFBHF with $p=\frac{1}{4(q+d)}$ decreases most rapidly.  Figure \ref{fig2} illustrates the decay of $E_k$ with the number of epochs for FBHF and VRFBHF, where $E_k=\frac{\|(x^{k+1}-x^k,u^{k+1}-u^k)\|}{\|(x^k,u^k)\|}.$  In particular, we consider the case $p=1$, which represents a partially deterministic variant of FBHB with random corrections. As shown in the Figure \ref{fig3}, Algorithm \ref{ALG} achieves a faster convergence in terms of CPU time for highlighting the computational advantage of the random correction. It can be seen that VRFBHF slightly outperforms FBHF when $d \geq q$.

\begin{figure}[h]
	
	\begin{minipage}{0.32\linewidth}
		\vspace{3pt}
		\centerline{\includegraphics[width=\textwidth]{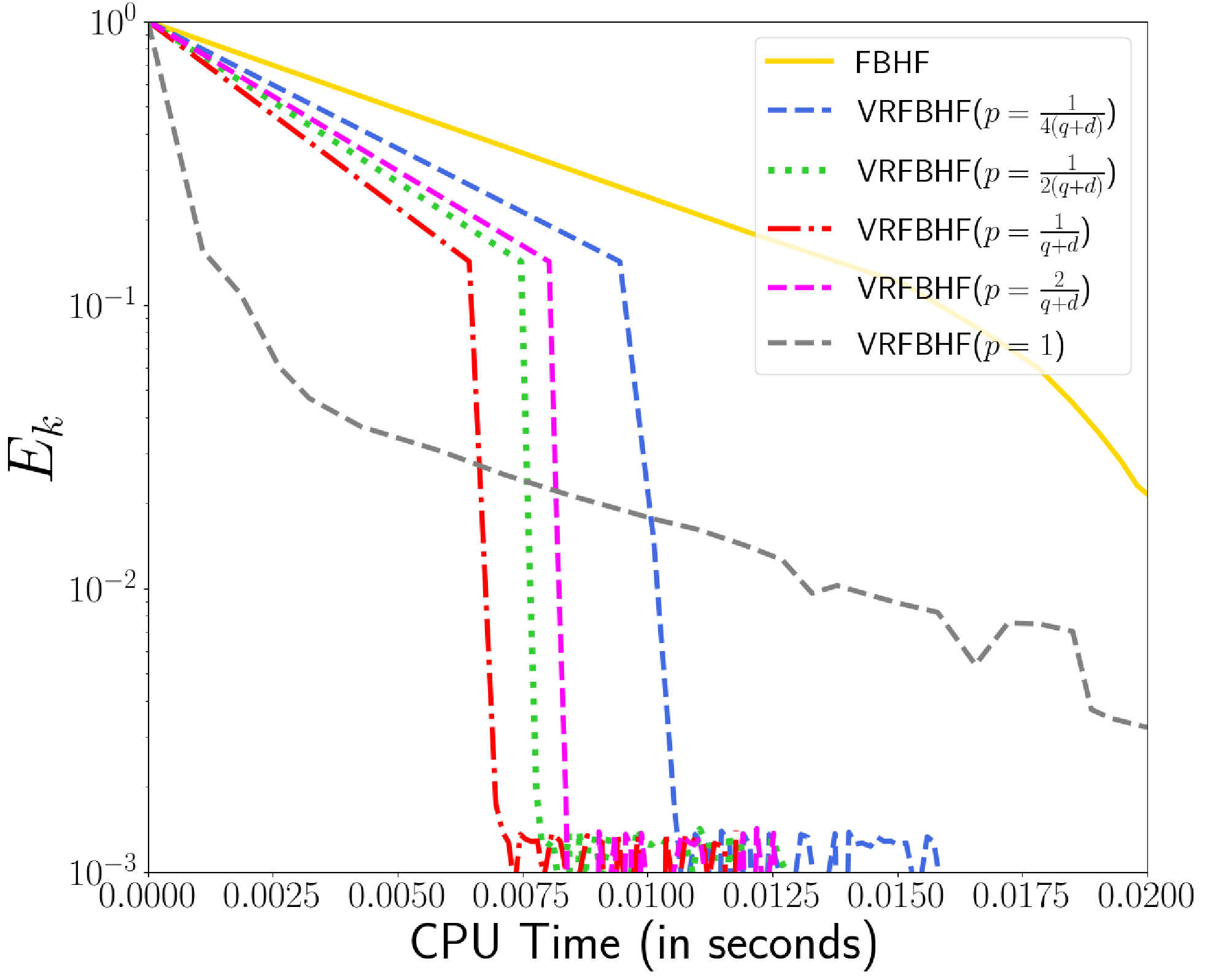}}
		\centerline{$q=1000,d=500$}
	\end{minipage}
	\begin{minipage}{0.32\linewidth}
		\vspace{3pt}
		\centerline{\includegraphics[width=\textwidth]{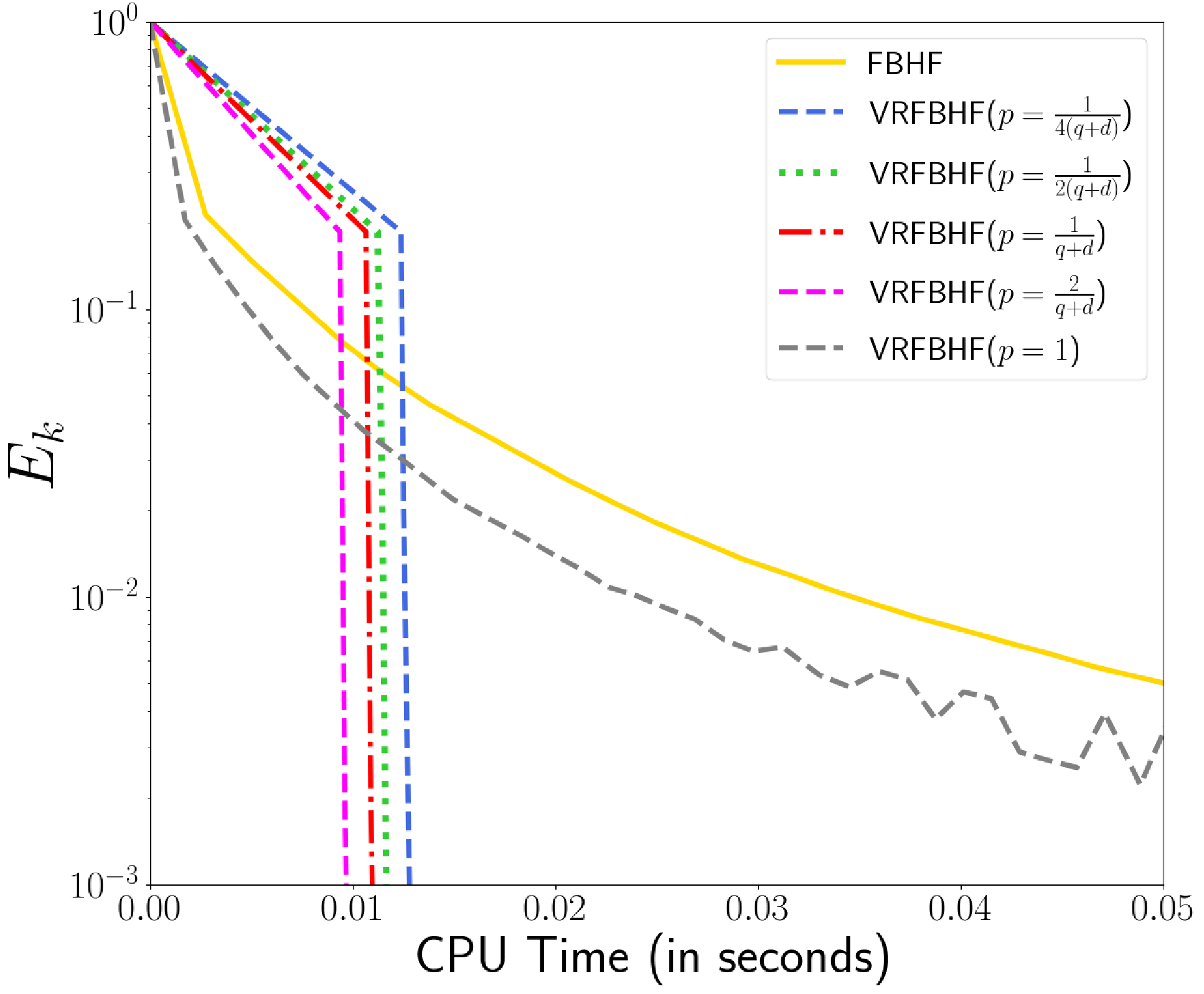}}	
		\centerline{$q=1000,d=1000$}
	\end{minipage}
	\begin{minipage}{0.32\linewidth}
		\vspace{3pt}
		\centerline{\includegraphics[width=\textwidth]{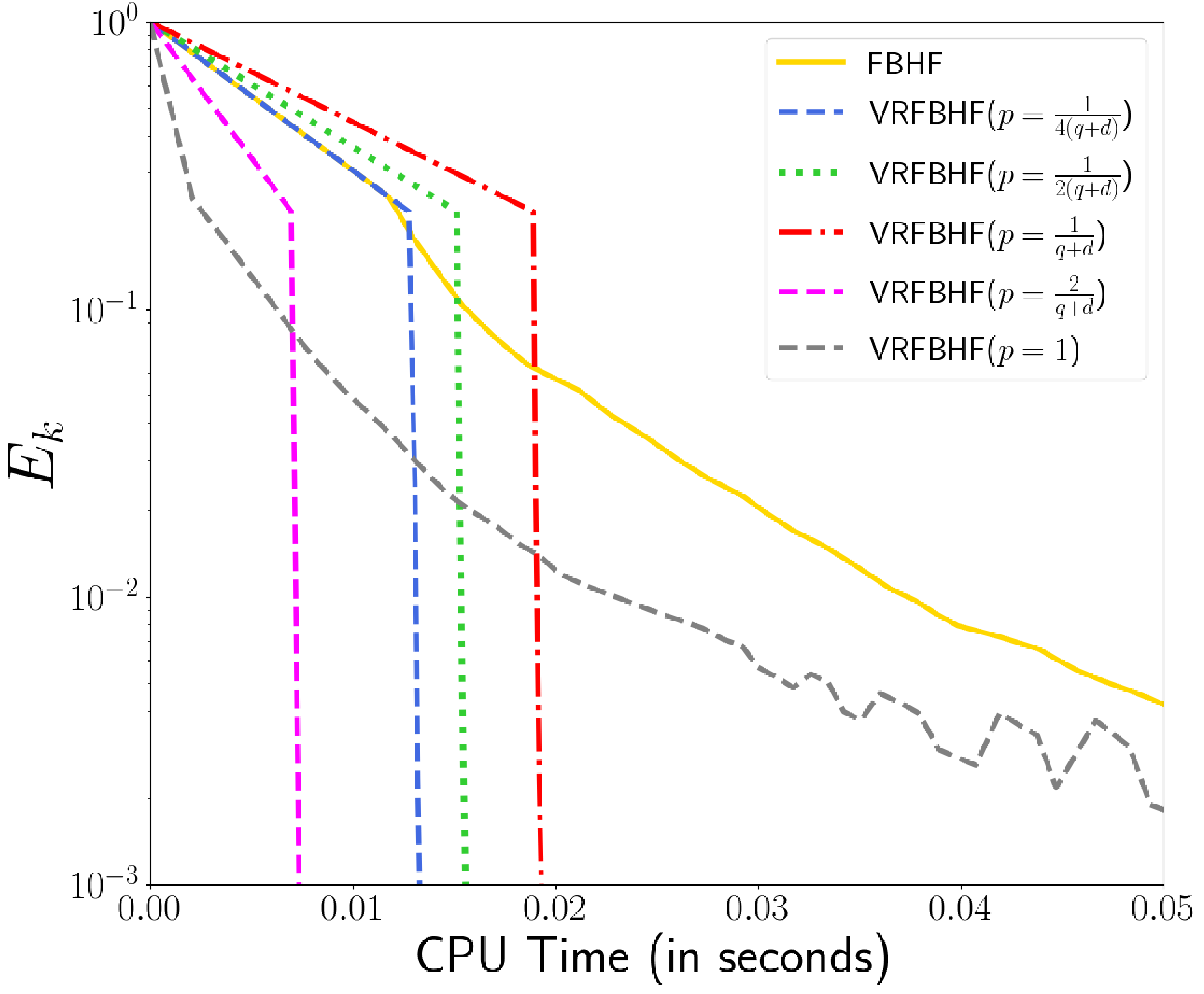}}	
		\centerline{$q=500,d=1000$}
	\end{minipage}

	\caption{Decay of $E_k$ with the CPU time}
	\label{fig3}
\end{figure}

\end{exm}
{
\subsection*{Acknowledgements}
We were deeply grateful to the anonymous reviewer for suggesting the generalization of the problem setting from finite-dimensional spaces to separable real Hilbert spaces and for their insightful suggestions that significantly enhanced the numerical experiments in the revised manuscript.
}
\subsection*{Declarations}
\textbf{Funding} The third author was supported by Scientific Research Project of Aeronautical Science Foundation of China (No.20200008067001) and National Natural Science Foundation
of China (No. 12271273).\\
\\
\textbf{Availability of data and materials}  The datasets generated during and/or analyzed during the current study are available
from the corresponding author on reasonable request.\\
\\
\textbf{Conflict of Interest} The authors declare that they have no conflict of interest.

\end{document}